\documentclass[letterpaper, 11pt]{article}

\usepackage{hyperref}
\usepackage{caption}
\usepackage{subcaption}
\usepackage{amscd,amssymb,amsfonts,amsmath,amsthm}
\usepackage{pstricks, pst-all}
\usepackage[all]{xy}
\usepackage{enumerate}
\usepackage{latexsym}
\makeindex

\newcommand{\colim}{\operatorname{colim}}
\newcommand{\sign}{\operatorname{sign\/}}

\newcommand{\ku}{k[u,u^{-1}]}

\newcommand{\C}{\operatorname{C\/}}
\newcommand{\Ab}{\operatorname{Ab\/}}
\newcommand{\Z}{\mathbb{Z\/}}

\newcommand{\CC}{\operatorname{CC\/}}

\newcommand{\HH}{\operatorname{HH\/}}
\newcommand{\MF}{\operatorname{MF\/}}

\newcommand{\id}{\operatorname{id\/}}
\newcommand{\RHom}{\operatorname{RHom\/}}
\newcommand{\RG}{\operatorname{R\Gamma\/}}

\newcommand{\Hom}{\operatorname{Hom\/}}

\newcommand{\Ext}{\operatorname{Ext\/}}

\newcommand{\Res}{\operatorname{Res\/}}

\newcommand{\cone}{\operatorname{cone\/}}
\newcommand{\ext}{\operatorname{ext\/}}
\newcommand{\inte}{\operatorname{int\/}}
\newcommand{\pr}{\operatorname{pr\/}}
\newcommand{\tR}{\widetilde{R}}

\def\H{\operatorname{H\/}}
\def\Spec{\operatorname{Spec\/}}
\def\stab{\operatorname{stab\/}}
\def\tr{\operatorname{tr\/}}

\def\D{{\operatorname{D\/}}}

\def\mod{\operatorname{mod\/}}

\def\op{\operatorname{op\/}}
\def\rank{\operatorname{rank\/}}
\def\Mat{\operatorname{Mat\/}}
\def\ch{\operatorname{ch\/}}
\def\T{\mathcal{T\/}}

\def\m{\mathfrak{m\/}}

\def\lra{\longrightarrow}
\def\lto{\longrightarrow}
\newcommand{\Zt}{\mathbb{Z\/}/2}

\newtheorem{theo}{Theorem}[section]
\newtheorem{lem}[theo]{Lemma}
\newtheorem{cor}[theo]{Corollary}

\newtheorem{prop}[theo]{Proposition}

\theoremstyle{definition}

\newtheorem{remark}[theo]{Remark}

\begin{document}

\title{The Kapustin-Li formula revisited}
\author{Tobias Dyckerhoff \and Daniel Murfet}
\maketitle

\begin{abstract} 
	We provide a new perspective on the Kapustin-Li formula for the duality pairing on the morphism complexes 
	in the matrix factorization category of an
	isolated hypersurface singularity. In our context, the formula arises as an explicit description of
	a local duality isomorphism, obtained by using the basic perturbation lemma and Grothendieck residues. 
	The non-degeneracy of the pairing becomes apparent in this setting.
	Further, we show that the pairing lifts to a Calabi-Yau structure on the
	matrix factorization category. This allows us to define
	topological quantum field theories with matrix factorizations as boundary conditions.
\end{abstract}

\tableofcontents

\section{Introduction}

Let $k$ be a field of characteristic zero and $R$ be a regular local augmented $k$-algebra with maximal ideal $\m$. 
We consider a function $w \in \m$ on $\Spec(R)$ which we assume to have an
isolated singularity at $\m$. The object of our interest is the local germ of the
singular hypersurface defined by the equation $w=0$. More precisely, we study the stable
homological algebra of the hypersurface by means of the category of matrix
factorizations $\MF(R,w)$. The latter is a differential $\Zt$-graded category
whose objects are given by $\Zt$-graded finite free $R$-modules $X = X^0 \oplus X^1$
equipped with an odd endomorphism $d$ satisfying $d^2 = w$. Such an object
$(X,d)$ corresponds, after choosing 
bases for $X^0$ and $X^1$, to a pair of
square matrices $(\varphi,\psi)$ satisfying 
\[
\varphi \circ \psi = \psi \circ \varphi = w \id \text{,}
\]
hence the nomenclature. Equivalently, we can combine the matrices into a
supermatrix
\[
Q = \left( \begin{array}{rr} 0 & \varphi \\ \psi & 0 \end{array} \right)
\]
satisfying $Q^2 = w \id$. We refer to $d$ or $Q$ as the twisted differential
associated with the matrix factorization $X$.
More details on the category $\MF(R,w)$ as well as an
overview of its relevance in terms of homological algebra over $R/w$ can be found in
\cite{eisenbud,yoshino,dyck4} and the references therein. 

Fixing a hypersurface $(R,w)$, we introduce the abbreviation $T =
\MF(R,w)$ as well as the symbol $[T]$ for the homotopy category of $T$. Recall
that the category $[T]$ is obtained by applying $\H^0(-)$ to all morphism complexes in $T$. For
matrix factorizations $X$ and $Y$, the morphism complex in $T$ will be denoted by
$T(X,Y)$, the morphisms in the homotopy category by $[T](X,Y)$.
As first established by Auslander \cite[Proposition 8.8 in Ch.~1 and Proposition 1.3 in Ch.~3]{auslander}, the triangulated category $[T]$ is a
Calabi-Yau category, i.e. there exist non-degenerate pairings
\[
[T](X,Y) \otimes_k [T](Y,X[n]) \to k
\]
for every pair of objects $X$,$Y$ in $[T]$. However, an explicit description of
this pairing was not known until this category appeared in the
context of topological string theory.
Following a proposal by Kontsevich, the physicists Kapustin and Li
\cite{kapustin}
interpreted the category $[T]$ as the category of boundary conditions in
the Landau-Ginzburg $B$-model corresponding to $(R,w)$. This allowed them to apply 
path integral methods when $k = \mathbb{C}$ in order to derive a formula
for a pairing
\[
[T](X,Y) \otimes_k [T](Y,X[n]) \to k, \;
(F,G) \mapsto 
\frac{1}{(2 \pi i)^n n!}\oint_{\left\{ | \partial_i w | = \epsilon \right\}} \frac{ 
\tr( FG (dQ)^{\wedge n})}{
\partial_1 w \partial_2 w \cdots \partial_n w}\text{,}
\]
where $Q$ is the twisted differential associated with $Y$.
A first attempt to put the pairing into a mathematical context was
outlined in \cite{segal}.

In \cite{murfet-2009} the second named author gave a mathematical derivation of
this formula and proved its non-degeneracy, as a special case of a general
statement about Serre duality in the singularity category of an arbitrary
isolated Gorenstein singularity. In this work we give an alternative and more
direct derivation of the pairing for hypersurfaces, using the techniques
developed in [ibid.]. Fundamentally the pairing is obtained from local
duality applied to the mapping complexes in the category $T$, and the explicit
form of the duality isomorphism is obtained by employing the basic perturbation
lemma as well as the theory of residue symbols. In this context the formula
naturally takes the form (Theorem \ref{kl})
\[
	(F,G) \mapsto 
	(-1)^{n+1 \choose 2}\frac{1}{n!}\Res \left[ \begin{array}{c} 
		\tr( FG (dQ)^{\wedge n}  )\\
		\partial_1 w, \partial_2 w,\; \cdots , \partial_n w \end{array}
		\right] \text{.}
\]

As a second main result, we show in Section \ref{sect.topfield} that the pairing given by the Kapustin-Li
formula is part of a Calabi-Yau structure on the dg category $\MF(R,w)$. This
means that the pairing factors canonically over the cyclic complex
of $\MF(R,w)$. The importance of this structure lies in the fact that it allows
us to define $2$-dimensional topological quantum field theories in the sense of \cite{costello} and \cite{lurie}.
This result is based on a variant of the calculation of the
boundary-bulk map in \cite{polishchuk} which we perform in Section \ref{sect.bb}. The representative 
\[	
 	\Hom(E,E) \to \Omega_w[n],\; F
		\mapsto (-1)^{n+1 \choose 2}\frac{1}{n!} \tr(  F(dQ)^{\wedge n}  )
\]
of the map which we
provide in Theorem \ref{theorem.bb} is adapted to the Kapustin-Li formula on the chain level.
Our argument involves another application of the basic perturbation lemma
where we use an explicit homotopy which contracts a Koszul complex onto its
cohomology. The construction of this canonical contracting homotopy in Section
\ref{sect.canhom} may be considered as a result of independent interest.

As a well-known application of the field theory formalism, we illustrate in
Section \ref{sect.rr} how a
Riemann-Roch formula, which presumably agrees with the one given in
\cite{polishchuk}, can be pictorially deduced from the existence of a field
theory.\\

We will now outline our derivation of the Kapustin-Li formula. To exhibit the 
relation to classical local duality, we sketch the
argument in a $\Z$-graded context, the detailed argumentation will be given in a
purely $\Zt$-graded setting. Fixing two objects
$X$,$Y$ in the category $T = \MF(R,w)$, 
we think of the mapping complex $T(X,Y)$ as a $2$-periodic $\Z$-graded complex.
Local duality provides an isomorphism
\begin{equation}\label{duality}
\RG_\m(T(X,Y)) \stackrel{\simeq}{\lra} \RHom( \Hom_R(T(X,Y),R), \RG_\m(R))
\end{equation}
in the derived category of $R$-modules. Analyzing the right hand side, observe that the graded trace pairing yields an isomorphism between the complex $\Hom_R(T(X,Y),R)$
and $T(Y,X)$. Since $R$ is a regular local ring, thus Gorenstein, we have an isomorphism $\RG_\m(R) \cong
\H^n_{\m}(R)[-n]$, and $\H^n_{\m}(R)$ is an injective hull of the residue field. Furthermore, there exists
a natural map 
\[
\Res: \H^n_{\m}(R) \lra k 
\]
given by the Grothendieck residue symbol.
Since the singularity is assumed to be isolated, the cohomology modules of the
complex $T(Y,X)$ have finite length which in turn implies that the map
\[
\Hom( T(Y,X), \H^n_{\m}(R))\stackrel{\Res_*}{\lra} \Hom( T(Y,X),R), k)
\]
is a quasi-isomorphism.
On the left hand side of the duality isomorphism (\ref{duality}), we deduce that the natural map
\[
\RG_\m(T(X,Y)) \stackrel{f}{\lra} T(X,Y)
\]
is a quasi-isomorphism since the restriction of $T(X,Y)$ to the complement of $\m$ in
$\Spec(R)$ is contractible. Combining the above observations, we obtain the diagram 
\[
\xymatrix@C=2cm{
\RG_\m(T(X,Y)) \ar[d]^f \ar[r]^(0.42)g &  \Hom(T(Y,X[n]), \H^n_{\m}(R)) \ar[d]^{\Res_*} \\
T(X,Y) \ar@{.>}[r] & \Hom( T(Y,X[n]), k)
}
\]
establishing the duality pairing of the category $T$. 
We subdivide the problem of finding an explicit formula for this
pairing into
\begin{enumerate}[(I)]
	\item Find a model of $\RG_m T(X,Y)$ in which the maps $f$ and $g$
		become explicit (\emph{Koszul model})
	\item Invert the map $f$ up to homotopy (\emph{Basic Perturbation Lemma})
	\item Describe the map $\Res_*$ explicitly (\emph{Grothendieck residues})\text{,}
\end{enumerate}
where we indicated the techniques which we will use in brackets.\\

In conclusion, the outline of the paper is as follows. After collecting the
necessary preliminaries in Section \ref{protagonists}, the new derivation of the
Kapustin-Li pairing is detailed in Section \ref{sect.kapli}, following the above
steps (I) through (III). Our calculation of
the boundary-bulk map is given in Section \ref{sect.bb} and then used in Section
\ref{sect.topfield} to construct topological quantum field theories and deduce a Riemann-Roch
formula.\\

\emph{Conventions.} We use the symbol $\cong$ to denote an isomorphism of
complexes and the symbol $\simeq$ to denote a quasi-isomorphism. Duals in
various contexts will be referred to by the symbol
$-^{\vee}$. Applied to an $R$-linear
complex $Z$, the complex $Z^{\vee}$ will be the mapping complex $\Hom(Z,R)$ with
the usual Koszul signs. Applied to an object $X$ in $\MF(R,w)$, we obtain an
object $X^{\vee}$ in $\MF(R,-w)$ by forming the mapping complex $\Hom(X,R)$
and ignoring the fact that the differential on $X$ does not square to $0$.
Similarly, $\Hom(-,-)$ always refers to a mapping complex where in the context
of matrix factorizations the differential does not necessarily square to $0$.
The $2$-periodic $\Z$-graded mapping complexes in the category $T = \MF(R,w)$ will
be denoted by $T(X,Y)$ whereas we use $\Hom(X,Y)$ for their $\Zt$-graded
counterpart. 

\section{Preliminaries} \label{protagonists}

\subsection{The Basic Perturbation Lemma}\label{sect.bpl}

Homological perturbation theory is concerned with the transport of algebraic
structures along homotopy equivalences of complexes. A typical example of a
structure which admits such a transport feature is the structure of an
$A_\infty$ algebra as introduced by Stasheff.
Roughly, the basic perturbation lemma is concerned with
transporting an additional differential $\delta$ on a complex $(B,d)$ along a homotopy equivalence of
complexes $f: (A,d) \to (B,d)$. One thinks of $d + \delta$ as a small perturbation of
$(B,d)$ and attempts to perturb $f$ and $(A,d)$ to obtain a new, perturbed
homotopy equivalence.
In comparison with spectral sequence techniques, the lemma has the advantage of producing explicit formulas.


We recall the variant of the basic perturbation lemma from \cite{crainic} which
we will apply. Let $R$ be a commutative ring with unit. 
A \emph{deformation retract datum} 
consists of 
\begin{equation}\label{eq:deforetractdatum}
\left[\xymatrix@C+1pc{
(A,d) \ar@<+.5ex>[r]^{\iota} & \ar@<+.5ex>[l]^{p} (B,d)\text{,} \; h }\right] \text{,}
\end{equation}
where $(A,d)$ and $(B,d)$ are complexes of $R$-modules, $\iota$ and $p$ are maps
of complexes, and $h$ is a homotopy on $B$ such that
\begin{enumerate}[(1)]
\item $p\iota = \id_A$
\item $\iota p = \id_B + dh + hd$.
\end{enumerate}
Given a perturbation of the differential on $B$, the lemma produces a new deformation retract.


\begin{lem}[Basic Perturbation Lemma] \label{basic} Suppose we are given a deformation
retract datum (\ref{eq:deforetractdatum}) and bounded below increasing
filtrations on $A$ and $B$ which are preserved by $\iota, p$ and $h$. Let
$\delta$ be a degree one map on $B$ which lowers the filtration and suppose that
$(d + \delta)^2 = 0$. Then the operator $\psi = \sum_{j \ge 0} (\delta h)^j \delta$
is well-defined and
\begin{itemize}
\item $\iota_{\infty} = \iota + h \psi \iota$, 
\item $p_\infty = p + p \psi h$,
\item and $h_{\infty} = h + h \psi h$
\end{itemize}
define a new \emph{perturbed} deformation retract datum
\begin{equation}\label{eq:deforetractdatum2}
\left[\xymatrix@C+1pc{
(A,d + p\psi \iota) \ar@<+.5ex>[r]^(0.5){\iota_{\infty}} & \ar@<+.5ex>[l]^(0.5){p_\infty} (B,
d + \delta)\text{,} \; h_{\infty} 
}\right] \text{.}
\end{equation}
\end{lem}
\begin{proof}
See \cite[Theorem 2.3]{crainic}.
\end{proof}


\subsection{The Koszul model for local cohomology}\label{sect.loc}

We present a quick derivation of some aspects of local cohomology which will be
relevant for us. Details can be found in \cite{residuesandduality,hartlocal,brunsherzog,kunz}.

Let $R$ be a regular local augmented $k$-algebra of Krull dimension $n$ with
maximal ideal $\m$.
For a finitely generated $R$-module $M$ we define the functor
\[
\Gamma_{\m}M = \{ x \in M \; :\; \m^k x = 0\;\; \text{for some $k\ge 0$}\}
\]
of global sections with support in $\{ \m \}$. Recall that the right derived
functors of $\Gamma_{\m}$ are the local cohomology functors with respect to $\m$
which we denote by $\H^i_{\m}(-)$. We can calculate local cohomology by using
the fact that there is a triangle
\begin{align}\label{triangle}
\RG_{\m} M \lra M \lra \RG(U,\widetilde{M}_{|U}) \lra \RG_{\m} M [1]
\end{align}
expressing $\RG_{\m} M [1]$ as the cone of the restriction map to the open
subscheme $U := \Spec(R) \backslash \{\m\}$ of $\Spec(R)$. 
Indeed, assume that $\underline{t} = \{ t_1, \dots, t_n \}$ is a system of
parameters for $R$. This system defines a covering $\mathfrak{U}$ of the punctured
spectrum $U$. The corresponding normalized \v{C}ech complex has graded
pieces
\begin{align}\label{eq.koszul}
\widetilde{C}^p({\mathfrak U},\, \widetilde{M}) = \bigoplus_{i_0 < \dots < i_p}
M_{t_{i_0} \dots t_{i_p}} 
\end{align}
with the usual differential given by the alternating sum over restriction maps.
It is well-known that the complex $\widetilde{C}^{\bullet}({\mathfrak U},\,
\widetilde{M})$ is quasi-isomorphic to $\RG(U,\widetilde{M}_{|U})$. The
restriction map $M \to \widetilde{C}^{\bullet}({\mathfrak U},\,
\widetilde{M})$ is given by the sum over the restriction maps $M \to M_{t_i}$. Using the above triangle
(\ref{triangle}), we can now calculate $\RG_{\m} M [1]$ and obtain
\[
\RG_{\m} M \simeq K^{\infty}(\underline{t}; M) := K^{\infty}(\underline{t}; R) \otimes_R M
\]
where 
\[
K^{\infty}(\underline{t}; R) := \bigotimes_{i=1}^{n} \left( R \lra R_{t_i}
\right)
\]
is called the stable cohomological Koszul complex corresponding to the sequence $\underline{t}$. We
remark that this argumentation generalizes in a straightforward way, when we
replace the module $M$ by a complex of $R$-modules.

\subsection{Generalized fractions}\label{sect.fractions}

As in the previous section, let $R$ be a regular local augmented $k$-algebra with maximal ideal
$\m$. Let $M$ be a finite $R$-module.
From the triangle (\ref{triangle}), we deduce the existence of a surjective map
\[
\xymatrix{
\H^{n-1}(U,\widetilde{M}) \ar@{->>}[r] &\H^n_{\m}(M)\text{.}
}
\]
Using the normalized \v{C}ech model for $\H^{\bullet}(U,\widetilde{M})$ described in the
previous section, this allows us to represent local cohomology classes by \v{C}ech
cocycles. More precisely, after choosing a system of parameters
$\underline{t}$, we obtain a surjective map
\[
\xymatrix{
M_{t_1 \cdots t_n} \ar@{->>}[r]^{\sigma_{\underline{t}}}&  \H^n_{\m}(M)
}
\]
and introduce the notation
\[
\left[ \begin{array}{c} m\\ t_1, t_2,\; \cdots, t_n \end{array} \right] :=
	\sigma_{\underline{t}}(\frac{m}{t_1 \cdots t_n}) \text{.}
\]
The expression on the left is called a generalized fraction. Note, that if the
denominator of a generalized fraction is changed, then the map
$\sigma_{\underline{t}}$ changes accordingly. Assume $\underline{t}'$ is a another system of parameters, such that
\[
t'_i = \sum_{j=1}^{n} C_{ij} t_j
\]
for a matrix $C$ with coefficients in $R$. Then we have the transformation rule
\begin{align}\label{trafo}
\left[ \begin{array}{c} m \\ t_1, t_2,\; \cdots, t_n \end{array} \right] =
\left[ \begin{array}{c} \det(C) m \\ t'_1, t'_2,\; \cdots, t'_n \end{array} \right]
	\text{.}
\end{align}
Detailed proofs of these statements can be found in \cite{lipman, kunz}.

\subsection{Dualizing functors and Grothendieck residues}\label{sect.residue}

We recall the dualizing theory developed in \cite[\textsection 4]{hartlocal}. Let $R$ be
as in the previous section and consider the category $\T$ of 
$R$-modules of finite length. A functor $D : \T^{\text{op}} \to \Ab$
into the category of abelian groups is called dualizing, if
\begin{enumerate}[(1)]
	\item $D$ is exact,
	\item $D(k) \cong k$.
\end{enumerate}
To a dualizing functor $D$ one associates the injective $R$-module 
\[
I = \colim_i D(R/\m^i)
\]
and proves that there exists a natural equivalence of
functors
\[
D(-) \stackrel{\cong}{\lra} \Hom_R(-,I) \text{.}
\]
One verifies that $I$ is an injective hull of the residue field $k$, and 
thus dualizing functors are unique (up to non-canonical equivalence).
In our situation, where $R$ is regular local, there are two natural dualizing
functors which we can consider.
\begin{enumerate}[(1)]
	\item The functor $\Hom_k(-,k)$ obviously defines a dualizing
		functor.
	\item The functor $\Ext^n(-,R)$ is a dualizing functor. Indeed,
		$\Ext^{\bullet}(k,R)$ is concentrated in degree $n$ and
		$\Ext^n(k,R) \cong k$, as one easily calculates via a Koszul
		resolution of $k$. But every module $M$ of finite length can be
		obtained via finitely many extensions by the module
		$k$. The long exact sequence for $\Ext$ tells us that
		$\Ext^{\bullet}(M,R)$ is concentrated in degree $n$, thus
		$\Ext^n(-,R)$ is exact. The injective module corresponding to
this dualizing functor is given by
		\[
		I = \colim_i \Ext^n(R/\m^i, R)\text{.}
		\]
		Directly from the definition of local cohomology, we deduce that
		$I$ is isomorphic to the top local cohomology module
		$\H^n_{\m}(R)$. In conclusion, we obtain a natural equivalence
		\[
		\Ext^n(-,R) \cong \Hom_R(- , \H^n_{\m}(R))
		\]
		of dualizing functors.
\end{enumerate}

\noindent
By the uniqueness of dualizing functors, there must exist a possibly non-canonical
equivalence between both functors.
After identifying $R$ with the rank $1$ free
$R$-module of top differential forms $\Omega^n$ (more precisely, we should use
universally finite differentials as explained in \cite{kunz.kaehler}), one can actually construct a
canonical identification via residues.
Namely, there exists a natural map
\[
\Res:\; \H^n_{\m}(\Omega^n) \lra k
\]
which is called the Grothendieck residue symbol. It induces an equivalence of
dualizing functors
\begin{align}\label{dualequiv}
	\Res_*:\; \Hom_R(-, \H^n_{\m}(\Omega^n)) \stackrel{\cong}{\lra}
	\Hom_k(-,k)\text{.}
\end{align}
Details on the construction of the residue
symbol and its natural properties can be found in \cite{lipman}. We will only recall how to calculate it in
terms of generalized fractions. Let us choose a regular sequence $\underline{x}$
of generators of the maximal ideal $\m$ in $R$. This yields a trivialization of
the module $\Omega^n$ by choice of the generator $d\underline{x} = dx_1 \wedge
\dots \wedge dx_n$.
Now let 
\[
\left[ \begin{array}{c} \omega\\ t_1, t_2,\; \cdots, t_n \end{array} \right] 
\]
be a generalized fraction representing an element of $\H^n_{\m}(\Omega^n)$ in
the sense of the previous section. Since $\underline{t}$ is a system of
parameters, there exists $i$ such that $\m^i \subset (t_1, \dots, t_n)$. Using
the transformation rule for generalized fractions, we find 
\[
\left[ \begin{array}{c} \omega\\ t_1, t_2,\; \cdots, t_n \end{array} \right] 
=
\left[ \begin{array}{c} r d\underline{x}\\ x_1^i, x_2^i,\; \cdots, x_n^i \end{array} \right] 
\]
for some $r \in R$ which can be calculated by formula (\ref{trafo}). We embed $R \subset k[
[x_1, \dots, x_n ]]$ and represent $r$ as a power series. Expanding
$\frac{r}{x_1^i \cdots x_n^i}$ as a Laurent series, the residue is given by the
coefficient corresponding to $(x_1 \cdots x_n)^{-1}$.

Finally, we mention that for $k =\mathbb C$ we can apply analytic methods to
calculate the residue symbol. In this case, we have
\[
\Res \left[ \begin{array}{c} \omega\\ t_1, t_2,\; \cdots, t_n \end{array} \right] = 
\frac{1}{(2 \pi i)^n}  \oint_{|t_i| = \varepsilon} \frac{\omega}{t_1 \dots t_n}
\]
and we refer to \cite{griffiths} for a detailed treatment of duality based on
this analytic definition.

\section{The Kapustin-Li Formula}\label{sect.kapli}

As above let $R$ be a regular local augmented $k$-algebra of finite Krull
dimension $n$ with maximal ideal $\m$. We fix a regular sequence
\[
\underline{x} = \{ x_1, x_2, \dots, x_n \}
\]
of generators of $\m$. We use the abbreviation $\partial_i :=
\frac{\partial}{\partial x_i}$ and introduce 
\[
\underline{t} = \{ \partial_1 w, \partial_2 w, \dots, \partial_n w \}
\]
which forms a system of parameters, since the singularity of $\Spec(R/w)$ is
assumed to be isolated and the characteristic of $k$ is zero (cf.
\cite[Proposition (1.2)]{looijenga}).
As previously, we abbreviate $T = \MF(R,w)$. Let $X,Y$ be matrix factorizations and consider the
morphism complex $T(X,Y)$ as a $2$-periodic $\Z$-graded complex of (free) $R$-modules.
In the introduction, we defined the diagram
\begin{align}\label{diagram}
\xymatrix@C=2cm{
\RG_\m(T(X,Y)) \ar[d]^f \ar[r]^(0.42)g &  \Hom(T(Y,X[n]), \H^n_{\m}(R)) \ar[d]^{\Res_*} \\
T(X,Y) \ar@{.>}[r] & \Hom( T(Y,X[n]), k)
}
\end{align}
and subdivided the explicit derivation of the duality pairing into three steps. 
In this section we will now provide the details. We choose to formulate the
argument in a purely $\Zt$-graded setting replacing the $2$-periodic mapping
complexes in $T(X,Y)$ by the $\Zt$-graded ones which we denote by $\Hom(X,Y)$. 
The comparison between the $\Zt$-graded and $\Z$-graded context is given as
follows. Define the $2$-periodification
\[
P : \; C^{\Zt}(R) \to C^{\Z}(R)
\]
which extends a $\Zt$-graded complex $2$-periodically and the $\Zt$-folding
\[
F : \; C^{\Z}(R) \to C^{\Zt}(R), Z \mapsto (\bigoplus_{k \text{ even}} Z^k) \oplus
(\bigoplus_{k \text{ odd}} Z^k) \text{.}
\]
Then observe that the following holds.

\begin{prop}\label{adjunction}
	Let $A$ be a $\Zt$-graded complex and $B$ a bounded $\Z$-graded
	complex. Then 
	\begin{enumerate}[(a)]
		\item we have
	\[
	\Hom^{\Z}_R(P(A), B) \cong P \Hom^{\Zt}_R(A, F(B))\text{,}
	\]
	in particular, after passing to homotopy classes of maps, we obtain
	\[
	[P(A), B] \cong  [A, F(B)]\text{.}
	\]
		\item Further we have
	\[
	P(A) \otimes_R^{\Z} B \cong P(A \otimes_R^{\Zt} F(B))\text{.}
	\]
	\end{enumerate}
\end{prop}

Thus we can translate diagram (\ref{diagram}) into the $\Zt$-graded setting in
virtue of the functors $P$ and $F$.

\subsection{(I) Koszul model}
We will use the Koszul model for the complex $\RG_\m(T(X,Y))$ in which both maps $f$ and $g$ become explicit. As explained in Section \ref{sect.loc}, this model is obtained as the tensor product
\begin{equation} \label{koszulmodel}
T(X,Y) \otimes_R K^{\infty}(t_1, t_2, \dots, t_n; R)\text{,}
\end{equation}
where $K^{\infty}(t_1, t_2, \dots, t_n; R)$ denotes the stable Koszul complex of a system of parameters $t_1, \dots, t_n$. 
As already pointed out above, since the singularity of $\Spec(R/w)$ is isolated, the
sequence of partial derivatives of $w$
\[
\underline{t} = \{ \partial_1 w, \partial_2 w, \dots, \partial_n w \}
\]
forms a system of parameters. We fix this system for the remainder of the
section. 
In view of Proposition \ref{adjunction}(b) we replace the complex (\ref{koszulmodel}) by the $\Zt$-graded
tensor product
\[
Z \otimes_R K\text{,}
\]
where $Z = \Hom_R(X,Y)$ denotes the $\Zt$-graded mapping complex and $K$ denotes
the $\Zt$-folding of the stable Koszul complex. Explicitly, we denote by
$(K,\delta)$ the $\Zt$-graded complex
\[
\bigotimes_{i=1}^n (R \to R_{t_i}\theta_i)
\]
where $\theta_i$ are odd bookkeeping variables. In the obvious way, we
introduce a graded-commutative multiplication on $K$. Then, the differential will 
simply be given by the left-multiplication
\[
\delta = \sum_i \theta_i\text{.}
\]
We think of the variables $\theta_i$ in $K$ as $1$-forms.

In this $\Zt$-graded context we will now make diagram
(\ref{diagram}) explicit. Even though local duality is of course the underlying motivation, we will not
apply any particular duality theorem, but rather reprove it explicitly in our specific situation.
The complex $Z \otimes_R K$ is a $\Zt$-graded model for $\RG_{\m}(T(X,Y))$, with respect to which we can give an explicit description of $f$ and $g$. 

We begin with the map
\[
g:\; \Hom(Y,X) \otimes_R K \lto \Hom(\Hom(Y,X[n]), \H^n_{\m}(R)) 
\]
which is obtained as a composition of various natural maps.
The tensor evaluation isomorphism is defined as
\begin{equation}\label{tensor-evaluation}
	\xi:\; Z \otimes_R K \lto \Hom(Z^{\vee}, K),\; F \otimes \omega \mapsto
	\left[ g \mapsto (-1)^{|g|(|\omega|+|F|)} g(F) \omega \right]\text{.}
\end{equation}
The cohomology of $K$ is concentrated in the $n$-form component, and we have
$\H^n(K) \cong \H^n_{\m}(R)[-n]$, so there is a quasi-isomorphism 
$\nu: K \to \H^n_{\m}(R)[-n]$.
The induced map
\[
\nu_*: \Hom(Z^{\vee}, K) \lra \Hom(Z^{\vee}, \H^n_{\m}(R)[-n])
\]
is a quasi-isomorphism as the complex $\cone(\nu_*)$ is acyclic. 
To see this, observe that the cone
of $\nu$ is the $\Zt$-folding of a bounded acyclic complex $C$. Thus, any
map from a $\Z$-graded complex $P$ into $C$ factors through a brutal truncation of
$P$ from above, and is therefore null-homotopic. The statement in the
$\Zt$-graded category now follows from Proposition \ref{adjunction}(a). 
Finally, there is a natural isomorphism of complexes
\[
\tau:\;	\Hom(Y,X) \lra \Hom(X,Y)^{\vee}, \; G \mapsto \tr(G \circ -).
\]
Here $\tr$ is the graded trace map, given by
\[
\tr: \; \Hom(X,X) \to R, \; \left( \begin{smallmatrix} A & B\\ C & D
\end{smallmatrix} \right) \mapsto \operatorname{trace}(A) -
\operatorname{trace}(D) \text{.}
\]
With this terminology, we have
\[
g = \tau^* \circ \nu_* \circ \xi \text{.}
\]
We now move on to study the map $f$.

\subsection{(II) Homotopy inverse of $f$} 

As above, we use the notation $(Z,d) = \Hom_R(X,Y)$ for the $\Zt$-graded mapping complex
in the category $\MF(R,w)$ and denote the $\Zt$-folded stable
Koszul complex by $(K,\delta)$ .
Thinking of the variables $\theta_i$ in $K$ as $1$-forms, the map 
\[
f:\; Z \otimes_R K \lto Z 
\]
is given by projection onto the $0$-form component. The following lemma is
well-known.

\begin{lem}\label{homotopies}
	Multiplication by $t_i = \partial_i w$ is null-homotopic on
	the complex $Z = \Hom_R(X,Y)$. If 
	\[
		Q = \left( \begin{array}{rr} 0 & \varphi\\ \psi & 0 \end{array} \right)
	\]
	represents the twisted differential of the matrix factorization $Y$,
	then postcomposition by $\partial_i Q$ provides a
	homotopy of $t_i$ with zero. In
	particular, the restriction of the complex $Z$ to $\Spec(R_{t_i})$ is
	contractible with contracting homotopy given by
	\[
	h_i = \frac{\partial_i Q}{t_i} \circ - \text{.}
	\]
\end{lem}
\begin{proof} The relation $Q^2 = w$ implies by the Leibniz rule
	\[
	\frac{\partial Q}{\partial x_i} Q + Q \frac{\partial Q}{\partial x_i} =
	\frac{\partial w}{\partial x_i} \id_X
	\]
	for every $1 \le i \le n$. This implies all assertions.
\end{proof}

Note that, ignoring the Koszul differential $\delta$ for now, we can split
\[
(Z \otimes_R K, d\otimes 1) \cong (Z,d) \oplus (Z \otimes \bigoplus_{\substack{i_1 < \dots <
i_l\\ l >0}} R_{t_{i_1} \cdots t_{i_l}}\theta_{i_1} \cdots \theta_{i_l}, d\otimes 1)
\]
where the right-hand side summand is contractible. Indeed, by 
combining the homotopies from Lemma \ref{homotopies} we can form the contracting
homotopy
\begin{align}\label{average}
h = \bigoplus_{i_1 < \dots < i_l} \frac{1}{l}\sum_{k=1}^l h_{i_k}
\text{.}
\end{align}
In other words, we obtain a deformation retract datum 
\[
\left[\xymatrix@C+1pc{
(Z,d) \ar@<+.5ex>[r]^(0.37){\iota} & \ar@<+.5ex>[l]^(0.63){f}
(Z \otimes_R K, d) \text{,}  \; -h
}\right]
\]
where $\iota$ is the canonical inclusion. 
Considering the differential $d + \delta$ on the complex $Z \otimes K$ as a perturbation of $d$, we apply Lemma \ref{basic} to otbain the perturbed deformation retract
\[
\left[\xymatrix@C+1pc{
(Z ,d) \ar@<+.5ex>[r]^(0.35){\iota_\infty} & \ar@<+.5ex>[l]^(0.65){f}
(Z \otimes_R K,d + \delta) \text{,}  \; h_\infty
}\right]
\]
where
\[
\iota_\infty = \iota + \sum_{j \ge 0} (-h)(\delta (-h))^j \delta \iota = \sum_{j
\ge 0} (-h \delta)^{j} \iota.
\]
In particular, $\iota_\infty$ represents the desired inverse to $f$ in the homotopy category of $R$-modules. 

With this calculation of the inverse in hand, we now pursue a concrete description of the composite $g \circ \iota_{\infty}$, which is the quasi-isomorphism
\begin{equation}\label{eq:composite_g_iota}
\xymatrix@C+1pc{
Z \ar[r]^(0.45){\iota_\infty} & Z \otimes K \ar[r]^(0.4){\xi} 
& \Hom( Z^{\vee} , K)\ar[r]^(0.4){\tau^* \circ \nu_*}  & \Hom(T(Y,X[n]), H^n_{\m}(R))}.
\end{equation}
Since this involves the projection $\nu$ the only relevant summand of
$\iota_{\infty}$ is the one mapping to the $n$-form component of $Z \otimes_R K$, given by $(-1)^{n}(h\delta)^n \iota$. More precisely, there is a commutative diagram
\begin{equation}\label{eq:composite_g_iota_2}
\xymatrix{
Z \otimes R_{t_1 \cdots t_n}\theta_1 \cdots \theta_n \ar[r] & Z \otimes
H^n_{\m}(R)[-n] \ar[d]^{\cong}\\
Z \ar[u]^{(-1)^n (h\delta)^n \iota} \ar[r]^(0.3){g \circ \iota_\infty} & \Hom(T(Y,X[n]), H^n_{\m}(R))
} \text{.}
\end{equation}
To evaluate $(h \delta)^n \iota$, recall that the differential $\delta$ on the
stable Koszul complex is given by the left-multiplication $\delta = \sum
\theta_i$.
Thus, we calculate
\begin{equation}\label{eq.composite}
\begin{split}
(-1)^n(h\delta)^n \iota (F) &= (-1)^n \frac{1}{n!} \sum_{\sigma \in S^n} 
\frac{\partial_{\sigma(1)} Q \theta_{\sigma(1)} \partial_{\sigma(2)} Q
\theta_{\sigma(2)} \cdots \partial_{\sigma(n)} Q\theta_{\sigma(n)}
F}{\partial_1 w \cdots \partial_n w} + \rho\\
&= (-1)^{n|F| + {n+1 \choose 2}} \frac{1}{n!} \sum_{\sigma \in S^n} \sign(\sigma)
\frac{\partial_{\sigma(1)} Q  \partial_{\sigma(2)} Q
\cdots
\partial_{\sigma(n)}Q 
F}{\partial_1 w \cdots \partial_n w}\theta_1\cdots\theta_n + \rho\text{.}
\end{split}
\end{equation}
Here the remainder $\rho$ consists of terms whose denominator is not divisible by
$\partial_1w \cdots \partial_nw$. Therefore, $\rho$ will be annihilated by the
residue map $\Res_*$ and can thus be neglected.

\subsection{(III) Grothendieck residues}
The final step of the argument will make use of Grothendieck's residue
symbol. 

\begin{lem} The cohomology modules of the mapping complex $\Hom_R(X,Y)$ are
	$R$-modules of finite length.
\end{lem}
\begin{proof}
	By Lemma \ref{homotopies} the partial derivatives $\partial_i
	w$ act trivially on the cohomology of the complex
	$\Hom(X,Y)$. Therefore, the cohomology modules are modules
	over the Milnor algebra
	\[
	\Omega_w \cong R / (\partial_1  w, \dots,
	\partial_n w) \text{.}
	\]
	However, the Milnor algebra is finite dimensional over $k$, since we
	assume that the singularity of $\Spec(R/w)$ is isolated. Because the
	cohomology modules are finitely generated $R$-modules, and thus finitely
	generated $\Omega_w$-modules, the claim follows.
\end{proof}
	
In combination with (\ref{dualequiv}) from Section \ref{sect.residue}, the lemma implies that the map $\Res_*$ in
the diagram (\ref{diagram}) is a quasi-isomorphism. In view of 
(\ref{eq.composite}) this leads us to the Kapustin-Li formula: the composition $\Res_* \circ g \circ \iota_{\infty}$ provides an explicit
quasi-isomorphism
	\[
	\Hom(X,Y) \stackrel{\simeq}{\lra} \Hom_k(\Hom(Y,X[n]),k)\text{.}
	\]
We reformulate this statement in terms of the corresponding pairing. 
We use the notation
	\[
	(dQ)^{\wedge n} = \sum_{\sigma \in S_n} \sign(\sigma) 
	\partial_{\sigma(1)} Q \cdots \partial_{\sigma(n)} Q\text{,}
	\]
	where $Q$ is the twisted differential corresponding to $X$.

\begin{theo}\label{kl}
The pairing
	\[
	\Hom(X,Y) \otimes_R \Hom(Y,X[n]) \to k,\; (F,G) \mapsto 
	(-1)^{n+1 \choose 2}\frac{1}{n!}\Res \left[ \begin{array}{c} 
		\tr(FG(dQ)^{\wedge n})\\
		\partial_1 w, \partial_2 w,\; \cdots , \partial_n w \end{array}  \right] 
	\]
	provides a homologically non-degenerate pairing on the morphism
	complexes of the category of matrix factorizations $\MF(R,w)$
	associated to the local germ of an isolated hypersurface singularity.
\end{theo}

\begin{proof}
We simply have to evaluate the composition 	
\[
\Res_* \circ g \circ \iota_{\infty}(F)(G)
\]
keeping track of the Koszul signs. Careful sign bookkeeping yields
\begin{align*}
\Res_* \circ g \circ \iota_{\infty}(F)(G) &= 
(-1)^{n|F| + { n+1 \choose 2} + 2n|G| + |F||G|} \frac{1}{n!} \Res \left[ \begin{array}{c} 
		\tr(G(dQ)^{\wedge n}F)\\
		\partial_1 w, \partial_2 w,\; \cdots , \partial_n w \end{array}
		\right] \\
		& = 
(-1)^{ n+1 \choose 2} \frac{1}{n!} \Res \left[ \begin{array}{c} 
		\tr(FG(dQ)^{\wedge n})\\
		\partial_1 w, \partial_2 w,\; \cdots , \partial_n w \end{array}
		\right]  \text{.}
\end{align*}
\end{proof}

We conclude this section with a comparison to the approach in \cite{murfet-2009}. For an isolated Gorenstein singularity $A$ the analogue of the category $\MF(R,w)$ is the stable category of maximal Cohen-Macaulay (CM) modules $\underline{\operatorname{CM}}(A)$, and following Buchweitz \cite{buchweitz} we identify this category with the homotopy category of acyclic complexes of finitely generated free $A$-modules. The equivalence sends a CM module to its complete free resolution which, viewed as an infinite sequence of matrices, generalises the notion of a matrix factorization (which may be viewed as a two-periodic complete free resolution). In \cite{murfet-2009} the perturbation lemma is applied to these complete free resolutions to obtain explicit complete injective resolutions, which give rise to a duality isomorphism in $\underline{\operatorname{CM}}(A)$ specialising to the Kapustin-Li formula when $A = R/w$.

In the present article we exploit the fact that in the hypersurface case we can apply local duality and the perturbation lemma directly to the morphism complexes $T(X,Y)$, which allows us to avoid the introduction of complete injective resolutions.

\section{The Boundary-Bulk Map} \label{sect.bb}

The work in this section should be seen as an addendum to the results in
\cite{polishchuk}. We establish an explicit formula for the boundary-bulk map
which is adapted to the Kapustin-Li formula in Theorem \ref{kl} on the chain level.

\subsection{Morita-theoretic construction}
We recall the definition of Hochschild homology in the context of To\"en's
derived Morita theory for dg categories (\cite{toen.morita, toen.lectures}). More precisely, we will use the $2$-periodic
variant defined in \cite[4.1]{dyck4}. Let $T$ be a $2$-periodic dg
category which we may consider as a module over $T^{\op}\otimes T$ via
\[
T^{\op}\otimes T \to C(\ku), (x,y) \mapsto T(x,y)
\]
for objects $x,y$ in $T$.
By \cite[Theorem 7.2]{toen.morita}, this map has a continuous extension
\[
\underline{\tr}: \widehat{T^{\op}\otimes T} \to C(\ku)
\]
which is unique up to homotopy. The induced map on homotopy categories yields
\[
[\underline{\tr}]: \D(T^{\op}\otimes T) \to \D(\ku)
\]
and the Hochschild homology of the category $T$ is defined to be $[\underline{\tr}](T)$.
Thus, the Hochschild homology is the trace of the identity functor which we may choose 
to think of as the dimension of $T$.

Now, every object $e$ in $T$, defines an object $(e,e)$ of $T^{\op}\otimes T$
which induces a representable functor
\[
h^{(e,e)}: T^{\op}\otimes T \to C(\ku), (x,y) \mapsto T(x,e) \otimes
T(e,y) \text{.}
\]
The composition law in $T$ provides us with a natural map
\begin{equation}
\label{natural.map}
\pi_e:\; h^{(e,e)} \to T
\end{equation}
in $\D(T^{\op}\otimes T)$ and thus we obtain an induced map
\begin{equation}\label{boundary-bulk}
[\underline{\tr}](\pi_e):\; T(e,e) \simeq [\underline{\tr}](h^{(e,e)}) \to [\underline{\tr}](T) \simeq \HH(T)
\end{equation}
in $\D(\ku)$, which we call the \emph{boundary-bulk map}.

The derived Morita theory developed by To\"en can be used in the context of
matrix factorization categories to calculate the Hochschild chain complex $\HH(T)$ (see \cite{dyck4}). In
recent work of Polishchuk and Vaintrob \cite{polishchuk}, an explicit formula for the
boundary-bulk map is calculated. We will provide a (homotopic) variant of the formula which 
is better adapted to the form of the Kapustin-Li pairing from Theorem \ref{kl}.
Following the method in [ibid.], we will describe the map in the context of
derived Morita theory.
The compatibility between the Kapustin-Li pairing and the boundary-bulk map 
will lead to the existence of an oriented $2$-dimensional
topological quantum field theory as discussed in Section \ref{sect.topfield}.

As in the previous sections, we fix the notation $T = \MF(R,w)$. We also introduce $\widetilde{R} =
R\otimes_k R$ and $\widetilde{w} = w(y) - w(x)$. Using the results of \cite{dyck4}, we have
an equivalence
\begin{equation}\label{morita.equiv}
	\D(T^{\op}\otimes T) \simeq \MF^{\infty}(\widetilde{R}, \widetilde{w})\text{.}
\end{equation}
Given a matrix factorization $E$ in the category $\MF(R,w)$, the
representable module which corresponds under the above equivalence to
$h^{(E,E)}$ is $E_x^{\vee} \boxtimes E_y$. By \cite[Corollary 5.4]{dyck4}, the identity
functor is represented by the stabilized diagonal $\Delta^{stab}$.
We have to identify the natural map
\[
\varphi_E:\; E_x^{\vee} \boxtimes E_y \to \Delta^{\stab}
\]
in $\MF^{\infty}(\widetilde{R}, \widetilde{w})$ which corresponds under the
equivalence (\ref{morita.equiv}) to the map $\pi_E$.
By \cite[Lemma 5.3]{dyck4}, there is a
quasi-isomorphism
\begin{equation} \label{lem53}
\Hom_{\widetilde{R}}(E_x^{\vee} \boxtimes E_y, \Delta^{stab}) \simeq \Hom_{\widetilde{R}}(E_x^{\vee} \boxtimes E_y, R)
\end{equation}
where $R$ is the diagonal $\widetilde{R} / \widetilde{w}$-module. The symbol
$\Hom_{\widetilde{R}}$ refers to the $\Zt$-graded $\widetilde{R}$-linear mapping complex and we
keep this convention throughout this section. The right-hand
side complex is in turn quasi-isomorphic to the $R$-linear mapping complex $\Hom_R(E,E)$.

\begin{lem}[\cite{polishchuk}] \label{lift} 
	The composition map $\pi_E$  corresponds to the unique class
	$\varphi_E$ in $\Hom_{\widetilde{R}}(E_x^{\vee} \boxtimes E_y, \Delta^{stab})$ which maps to $\id_E$
	under the quasi-isomorphism (\ref{lem53}).
\end{lem}
\begin{proof}
	The composition functor
	\[
	T(-,E) \otimes T(E,-) \to T(-,-)
	\]
	is uniquely characterized by the property that it maps $\id \otimes \id$ in
	$T(E,E) \otimes T(E,E)$ to the identity in $T(E,E)$. 
	We interpret this statement in the category $\MF^{\infty}(\widetilde{R},
	\widetilde{w})$: The map
	\[
	\varphi_E:\; E_x^{\vee} \boxtimes E_y \to \Delta^{\stab} 
	\]
	is characterized by the property that the induced map
	\[
	(\varphi_E)_*:\; \Hom_{\widetilde{R}}(E_x^{\vee} \boxtimes E_y,E_x^{\vee} \boxtimes E_y) \to
	\Hom_{\widetilde{R}}(E_x^{\vee} \boxtimes E_y, \Delta^{\stab}) \cong \Hom_R(E,E),\; f
	\mapsto \varphi_E \circ f
	\]
	maps $\id_{E_x^{\vee} \boxtimes E_y}$ to $\id_E$. But this proves
	the claim, since $(\varphi_E)_*(\id) = \varphi_E$.
\end{proof}

To find $\varphi_E$ we therefore have to find an explicit homotopy inverse of
the quasi-isomorphism (\ref{lem53}). Again, we will use the basic perturbation
lemma to provide a solution to this problem.
We first introduce some notation.
Recall (\cite[2.3]{dyck4}) that we may represent 
\[
\Delta^{\stab} = (\tR\left<\theta_1, \dots, \theta_n\right>, \iota_{\Delta} +
\epsilon_{\lambda})
\]
where $\widetilde{w} = w_1 \Delta_1 + \dots + w_n \Delta_n$, $\iota_{\Delta}$ is
contraction with the element $\Delta_1 \theta_1^{\vee} + \dots + \Delta_n
\theta_n^{\vee}$ and $\epsilon_{\lambda}$ denotes exterior left multiplication with the
element 
\begin{equation}\label{1form}
	\lambda = w_1 \theta_1 + \dots + w_n \theta_n \text{.}
\end{equation}
Note that the coefficients
$\{w_i\}$
are not unique, different choices of $\{w_i\}$ will lead to different (but isomorphic)
models of $\Delta^{\stab}$. As explained below we will make a specific canonical
choice for these coefficients during the calculation. 
We reformulate
\[
\Hom_{\widetilde{R}}(E_x^{\vee} \boxtimes E_y, \Delta^{\stab}) \cong \Hom_{\widetilde{R}}(E_y, E_x) \otimes_{\tR}
\tR \left< \theta_1, \dots, \theta_n\right>
\]
which allows us to think of elements in $\Hom_{\widetilde{R}}(E_x^{\vee} \boxtimes E_y,
\Delta^{\stab})$ as supermatrix-valued differential forms. For a homogeneous
element
\[
A \otimes \eta \in \Hom_{\widetilde{R}}(E_y, E_x) \otimes_{\tR} \tR \left< \theta_1, \dots, \theta_n\right>
\]
we define
\begin{align*}
	d_Q(A \otimes \eta) & = (Q_x A - (-1)^{|A|} A Q_y) \otimes \eta\\
	\iota_{\Delta}(A \otimes \eta) & = (-1)^{|A|} A \otimes
	\iota_{\Delta}(\eta)\\
	\epsilon_{\lambda} (A \otimes \eta) & = (-1)^{|A|} A \otimes
	\epsilon_{\lambda}(\eta) \text{.}
\end{align*}
Here, we use the notation $Q_x = Q(x)$ and $Q_y = Q(y)$ to minimize the number of brackets. The differential on the complex
\[
\Hom_{\widetilde{R}}(E_y, E_x) \otimes_{\tR} \tR \left< \theta_1, \dots, \theta_n \right>
\]
is then given by $d_Q + \epsilon_{\lambda} + \iota_{\Delta}$. Note that we may
interpret this differential as a perturbation of the Koszul differential
$\iota_{\Delta}$ by $\delta = d_Q + \epsilon_{\lambda}$. This allows us to apply
the basic perturbation lemma. Indeed, the Koszul complex
\begin{equation} \label{koszul}
(\tR \left< \theta_1, \dots, \theta_n \right>, \iota_{\Delta})
\end{equation}
has cohomology $R$ concentrated in degree $0$ which allows us to define the
deformation retract
\[
\left[\xymatrix{
 (R,0) 
\ar@<+.5ex>[r]^(0.3){\iota} & \ar@<+.5ex>[l]^(0.7){p} 
(\tR \left< \theta_1, \dots, \theta_n \right>,
\iota_{\Delta})\text{,} \; -H }\right] \text{.}
\]
Here, $p$ is the projection onto the cohomology, $\iota$ is the inclusion of $R$ into
the first component of $\widetilde{R}$ and $H$ is a $k$-linear homotopy which contracts the complex (\ref{koszul}) onto its
cohomology. Observe that we have an isomorphism of graded vector spaces
\[
\Hom_{\widetilde{R}}(E_y, E_x) \otimes_{\tR} \tR \left< \theta_1, \dots, \theta_n \right>
\cong \Mat_n(k) \otimes_k \tR \left< \theta_1, \dots, \theta_n \right>
\]
where $r = \rank(E)$ and $\Mat_r(k)$ denotes the $\Zt$-graded vector space of
$r$-by-$r$ super matrices. This allows us to extend the above deformation
retract linearly to obtain the retract
\[
\left[\xymatrix{
(\Hom_{\widetilde{R}}(E_y,E_x) \otimes_{\tR} R, 0) 
\ar@<+.5ex>[r]^(.37){\iota} & \ar@<+.5ex>[l]^(.63){p} 
(\Hom_{\widetilde{R}}(E_y, E_x) \otimes_{\tR} \tR \left< \theta_1, \dots, \theta_n \right>,
\iota_{\Delta})\text{,} \; -H }\right] \text{.}
\]
The basic perturbation lemma allows us to perturb this retract by $\delta = d_Q
+ \epsilon_{\lambda}$ to obtain the following result.

\begin{prop}\label{bb.homotopyinverse}
	A homotopy inverse of the projection map 
	\[
	(\Hom_{\widetilde{R}}(E_y, E_x) \otimes_{\tR} \tR \left< \theta_1, \dots, \theta_n
	\right>, d_Q + \iota_{\Delta} + \epsilon_{\lambda})  \to (\Hom_R(E, E),
	d_Q)
	\]
	is given by the map
	\[
	\iota_{\infty} = \sum_{k=0}^n (-H \delta)^k \iota \text{.}
	\]
\end{prop}

To find an explicit formula for $\iota_{\infty}$, we thus have to construct an
explicit contracting homotopy $H$ of the Koszul complex (\ref{koszul}). We will
do this in the next section, but before that, let us conclude how to obtain
the boundary-bulk map from $\iota_{\infty}(\id)$.

\begin{lem} \label{bulk.lemma}
	Let $\eta = \iota_{\infty}(\id)$ with $\eta_n = B \theta_1 \dots \theta_n$.
	Then the boundary-bulk map is given explicitly by
	\[
	[\underline{\tr}](\pi_E): \; \Hom_R(E,E) \to \Omega_{w}[n],\; F \mapsto
	(-1)^{n|F|} \tr( B F ) \text{.}
	\]
\end{lem}
\begin{proof}
	The map 
	\[
	\varphi_E:\; \Hom_{\widetilde{R}}(E_x, E_y) \to \Delta^{\stab} 
	\]
	corresponding to $\eta$, is given by
	\begin{equation}\label{traceeta}
	F \mapsto (-1)^{n|F|}\tr(B F) \theta_1 \cdots \theta_n + l \text{.}
	\end{equation}
	Here the sign is contributed by the tensor evaluation map (cf. 
	\ref{tensor-evaluation}) and $l$ consists of terms involving $k$-forms
	with $k< n$ which are, as we will see, irrelevant.
	By Lemma 5.9 in \cite{dyck4}, we deduce that we have
	\[
	\underline{\tr}: \;\MF(\widetilde{R}, \widetilde{w}) \to C(\ku),\; X
	\mapsto \Hom_{\widetilde{R}}(\Delta^{\stab}_{\widetilde{w}}[n], X) \cong
	(\Delta^{\stab}_{\widetilde{w}}[n])^{\vee} \otimes_{\widetilde{R}} X \text{.}
	\]
	Now, directly from the definition of $\Delta^{\stab}$, it is easy to see
	that $(\Delta^{\stab}_{\widetilde{w}}[n])^{\vee}$ is isomorphic to
	$\Delta^{\stab}_{-\widetilde{w}}$. By the perturbation argument in the
	proof of Proposition
	\ref{bb.homotopyinverse} we deduce in complete analogy that the projection map
	\[
	\Delta^{\stab}_{-\widetilde{w}} \otimes_{\widetilde{R}} X \lto R \otimes_{\widetilde{R}} X\\
	\]
	has a homotopy inverse, in particular it is a quasi-isomorphism. 
	Thus, we have an equivalence $[\underline{\tr}] \simeq R \otimes_{\widetilde{R}} -$.
	The cohomology of the Hochschild chain complex $R
	\otimes_{\widetilde{R}} \Delta^{\stab}$ is concentrated in degree $n$
	(i.e. the parity of $n$ due to the $\Zt$-grading) where it is isomorphic to the Milnor
	algebra $\Omega_w$. More precisely the cohomology is concentrated in the $n$-form
	component. 
	Therefore, projection from the Hochschild complex onto $\Omega_w
	\theta_1 \dots \theta_n$ is a quasi-isomorphism. In view of
	(\ref{traceeta}), postcomposing
	$[\underline{\tr}](\varphi_E)$ with this projection results in the
	asserted formula.
\end{proof}

Note that this lemma corresponds to \cite[3.1.1]{polishchuk}.

\subsection{Canonical contracting homotopy of the Koszul complex}
\label{sect.canhom}

In this section, we construct an explicit canonical homotopy which contracts the 
Koszul complex $K(\Delta_1, \Delta_2, \dots, \Delta_n)$ onto its cohomology.
This will serve the purpose of finding an explicit expression for the homotopy
inverse $\iota_{\infty}$ from Proposition \ref{bb.homotopyinverse}. Aside from
that the result may be considered interesting in its own right.

\begin{remark} \label{remark.complete} To simplify the notation, we will construct the homotopy over the polynomial
ring $\widetilde{R} = k[x_1,y_1,\dots,x_n,y_n]$. The argumentation carries over
verbatim to the corresponding power series ring which we are really interested
in. Indeed, one simply has to replace all tensor products over $k$ by completed tensor
products and extend the maps continuously.
\end{remark}

For each $1 \le i \le n$, consider the augmented Koszul complex of the (length
$1$) sequence $\Delta_i$ in $k[x_i,y_i]$
\[
\xymatrix{
k[x_i,y_i] \theta_i \ar[r]^{d_i} & k[x_i,y_i]
\ar[r]^{d_i} & k[x_i] \xi_i\text{.}
}
\]
So we have
\[
d_i(f \theta_i) = f \Delta_i
\]
and
\[
d_i(f) = f\; (\mod \Delta_i) \text{.}
\]
We have canonical contracting homotopies $h_i$ which are defined as follows.
An element $f \in k[x_i,y_i]$ can be uniquely written as
\[
f = f_0 + \Delta_i f_1 \quad \text{with $f_0 \in k[x_i]$,}
\]
and we define
\begin{equation} \label{one.homotopy}
h_i:\; k[x_i,y_i] \to k[x_i,y_i]\theta_i, \; f \mapsto f_1 \theta_i 
\end{equation}
which we may think of as division by $\Delta_i$ without remainder.
Furthermore, we let
\[
h_i:\; k[x_i] \xi_i \to k[x_i,y_i],\; f \xi_i \mapsto f
\]
be the inclusion. The variables $\xi_i$ and $\theta_i$ are graded commutative bookkeeping variables of degree $1$ and
$-1$ respectively. 
We define the \emph{extended Koszul complex} $EK$ of the sequence $\{\Delta_1,
\dots, \Delta_n\}$ to be the tensor product over $k$ of the augmented Koszul
complexes. Using the Koszul signs rule one can easily check that the map
\[
h = \frac{1}{n}(h_1 + h_2 + \dots + h_n)
\]
defines a contracting homotopy on the extended Koszul complex. Note that as
graded vector spaces, we have
\[
EK = E \oplus K
\]
where $K$ is the graded space underlying the usual Koszul complex of $\{\Delta_1,
\dots, \Delta_n\}$. In terms of
the bookkeeping variables, $K$ consists of those elements which do not have any
$\xi_i$ terms. We call $K$ the \emph{interior}, $E$ the
\emph{exterior} of $EK$. The picture we have in mind is a half-open hypercube
whose faces constitute $E$. The $i$-th face in $E$ is given by those elements which
are multiples of a $\xi_i$. Note that each face is a Koszul complex of one
variables less.

We would like to use the contracting homotopy on $EK$ to define one on $K$. The
issue is, however, that even though $K$ is stable under $h$, it is not stable
under the differential $d$ on $EK$. Nevertheless, we can construct a
canonical perturbation of $h$ which provides an explicit contracting homotopy of
$K$. To this end, we introduce one last bit of notation. We define the maps
\[
\pr_i: K(\Delta_1, \Delta_2, \dots, \Delta_n) \to k[x_i] \otimes_k
K(\Delta_1,\dots, \widehat{\Delta_i}, \dots, \Delta_n)
\]
where $\pr_i \omega$ is obtained from $\omega$ via substituting $y_i$ by $x_i$
and removing all terms which are multiples of $\theta_i$. Note that we can naturally think of
the right-hand side Koszul complex as a subcomplex of $K(\Delta_1, \Delta_2,
\dots, \Delta_n)$. This allows us to abuse notation and consider the element
$\pr_i \omega$ as an element of $K(\Delta_1, \Delta_2, \dots, \Delta_n)$. One
easily verifies that $\pr_i$ is a map of complexes.

\begin{lem}\label{canonical.homotopy}
	There exists a unique family $\{H^{(n)} |\; n \ge 1 \}$ of homotopies 
	$H^{(n)}$ of the Koszul complexes $K(\Delta_1, \Delta_2, \dots, \Delta_n)$ satisfying
	\begin{enumerate}
		\item The homotopy $H^{(1)}$ agrees with the one defined in
			(\ref{one.homotopy}).
		\item For $n > 1$ we have the recursive formula
			\begin{equation}\label{recursion}
				H^{(n)} = h +
				\frac{1}{n} \sum_i H^{(n-1)} \circ \pr_i
				\text{.}
			\end{equation}
	\end{enumerate}
	Each homotopy $H^{(n)}$ contracts the corresponding Koszul complex onto its
	cohomology.
\end{lem}

\begin{proof}
	We argue by induction on $n$.
We decompose the differential $d$ on $EK$ into $d = d_{\ext} + d_K$, where 
$d_{\ext} = \pr_E \circ d$ and $d_K = \pr_K \circ d$. Similarly, the homotopy
$h$ does
not preserve $E$ and we have $h = h_E + h_{\inte}$ with $h_{E} = \pr_E \circ
h$ and $h_{\inte} = \pr_K \circ h$. First, let $\omega$ be an element
of negative degree in $K$. Then we have
\begin{equation}\label{homotopy.formula}
	\begin{split}
\omega & = [d,h] \omega \\
& = d_{\ext} h \omega + d_K h \omega + h d_K \omega + h d_{\ext} \omega \\
& = [d_K, h] \omega + h_{\inte} d_{\ext} \omega 
\end{split}
\end{equation}
where the last equality follows since $\omega$ lies in $K$ and so all exterior
components must cancel out. 
Directly from the definitions we calculate
\begin{equation}\label{homotopy.middle}
h_{\inte} d_{\ext} \omega = \frac{1}{n}( \pr_1 \omega + \pr_2 \omega + \dots +
\pr_n \omega) \text{.}
\end{equation}
Assume, $H^{(n-1)}$ is a contracting homotopy for the Koszul complex in $n-1$
variables.
Then, we calculate
\begin{align*}
	[d_K,H^{(n)}] \omega & = [d_K, h] \omega + [d_K, \frac{1}{n}\sum_i
	H^{(n-1)} \circ \pr_i] \omega \\
	& = [d_K, h] \omega + \frac{1}{n} (\pr_1 \omega + \pr_2 \omega + \dots +
\pr_n \omega) && \text{(since $d_K$ commutes with $\pr_i$)}\\
	& = \omega &&  \text{(use (\ref{homotopy.middle}) and
	(\ref{homotopy.formula}))}
\end{align*}
Now let $f$ be an element of degree $0$ in $K$. Define the augmentation maps
\[
p^{(n)}: \; \widetilde{R} \to \widetilde{R}/(\Delta_1, \dots, \Delta_n) \cong
k[x_1,\dots,x_n]
\subset \widetilde{R} \text{.}
\]
We have to show that 
\[
d_K H^{(n)} f + p^{(n)} f = f \text{.}
\]
Again, we argue inductively and calculate
\begin{align*}
d_K H^{(n)} f + p^{(n)} f & = d_K h f + (\frac{1}{n}\sum_j d_K H^{(n-1)} \pr_j f) + p^{(n)}
f\\
& = d_K h f + \frac{1}{n}\sum_j (d_K H^{(n-1)} \pr_j f + p^{(n-1)} \pr_j f)\\
& = \frac{1}{n} (\sum_j (d_j h_j f +\pr_j f))\\
& = \frac{1}{n} (\sum_j f ) = f \text{.}
\end{align*}
This proves all assertions.
\end{proof}

We give an explicit formula for $H^{(n)}$.

\begin{cor}\label{exp.formula}
	Explicitly, we have
	\[
	H^{(n)} = (h_1 + h_2 + \dots + h_n) \circ P^{(n)}
	\]
	where 
	\begin{equation} \label{explicit.formula}
		\begin{split}
		P^{(n)} & = \sum_{l = 0}^{n-1} a(l) \sum_{j_1 < j_2 < \dots <
		j_{l}} \pr_{j_1} \circ \pr_{j_2} \circ \dots \circ \pr_{j_l}\\
		& = \frac{1}{n} \id + \frac{1}{n(n-1)}\sum_j \pr_j + \dots
		\end{split}
	\end{equation}
	and 
	\[
	a(l) = \frac{1}{n-l}{n \choose l}^{-1}
	\text{.}
	\]
\end{cor}
\begin{proof} This can be easily deduced from the recursive properties of
	$H^{(n)}$.
\end{proof}

We will never actually make use of this explicit formula in our calculation. In fact, we
will only need various simple properties of $H^{(n)}$ and its
components which are collected in the following proposition. 
For a $k$-form $\omega$ in 
$K(\Delta_1, \Delta_2, \dots, \Delta_n)$, we let $\omega(y_i)$ be the $k$-form obtained from $\omega$ via 
substituting $x_i$ by $y_i$. Analogously, we define $\omega(x_i)$. Note the
difference between $\pr_i \omega$ and $\omega(x_i)$.

\begin{prop}\label{properties}
	\begin{enumerate}[(1)]
		\item \label{p2} The homotopies $h_i$ and $h_j$ anticommute for
			all $i,j$.
		\item \label{p1} For $\omega$ in $K(\Delta_1, \Delta_2, \dots, \Delta_n)$ we have
			\[
			h_i(\omega(y_i)) \equiv \theta_i
			\frac{\partial}{\partial x_i} \omega(x_i) \quad
			\text{$(\mod \Delta_i)$.}
			\]
		\item \label{p3} We have
			\[
			h_i \circ P^{(n)} = P^{(n)} \circ h_i
			\]
			for the iterated projection map from Corollary
			\ref{exp.formula}.
		\item \label{p4} Let $k >0$ and let $\omega$ be a $k$-form in $K(\Delta_1, \Delta_2, \dots, \Delta_n)$. 
			Then we have
			\[
			P^{(n)}(\omega) \equiv \frac{1}{k} \omega \quad \text{$(\mod \Delta_1,\dots,\Delta_n)$.}
			\]
	\end{enumerate}
\end{prop}
\begin{proof} 
	\begin{enumerate}[(1)]
\item 
	This is immediate from the construction.
\item 
	Expanding the coefficients of $\omega(y_i)$ into its Taylor series with
	respect to the variable $x_i$, we obtain
	\[
	\omega(y_i) = \omega(x_i) + \Delta_i \frac{\partial}{\partial x_i} \omega(x_i) 
	+ \frac{\Delta_i^2}{2} \left(\frac{\partial}{\partial x_i}\right)^2 \omega(x_i) + \dots \text{,}
	\]
	which implies the assertion.
\item 
	The projection maps $\pr_i$ commute with $h_j$ for all $i,j$. Note that $\pr_i h_i = h_i \pr_i = 0$.
\item 
	For $k = n$, we have 
	\[
	P^{(n)} \omega = \frac{1}{n} \omega
	\]
	since all maps $\pr_j$ annihilate $\omega$. We proceed by induction on
	$n-k$. The iterative projection map satisfies the recursion
	formula
	\[
	P^{(n)} = \frac{1}{n}( \id + \sum_j P^{(n-1)} \circ \pr_j)
	\]
	which allows us to proceed by induction on $n-k$. It suffices to prove
	the statement for $\omega = f \theta_{i_1} \theta_{i_2} \cdots
	\theta_{i_k}$. In this case, we have 
	\begin{align*}
	P^{(n)}\omega & = \frac{1}{n}( \omega + \sum_j P^{(n-1)} \circ \pr_j)\\
	& = \frac{1}{n}( \omega + \sum_{j \ne i_l} P^{(n-1)} f(x_j) \theta_{i_1} \theta_{i_2} \cdots
	\theta_{i_k} )\\
	& = \frac{1}{n}( \omega + (n-k) \frac{1}{k} \omega) \quad \text{$(\mod
	\Delta_1,\dots,\Delta_n)$}\\
	& = \frac{1}{k} \omega \quad \text{$(\mod \Delta_1,\dots,\Delta_n)$}
	\end{align*}
	\end{enumerate}
	which proves our claim.
\end{proof}

\subsection{Explicit Formula}

In this section, we set $R = k[ [x_1, \dots, x_n] ]$
and work with completed tensor products (cf. Remark \ref{remark.complete}). We will now use the canonical Koszul homotopy to calculate $\eta =
\iota_{\infty}(\id)$ with $\iota_{\infty}$ from Proposition
\ref{bb.homotopyinverse}. 
In fact, in view of Lemma \ref{bulk.lemma}, we only
need an explicit formula for $\eta_n$ modulo $(\Delta_1, \Delta_2, \dots,
\Delta_n)$.
We fix $n$ and denote the canonical contracting homotopy $H^{(n)}$ and the
iterative projection map $P^{(n)}$ of the
previous section by $H$ and $P$ respectively.

At this point, we choose the $1$-form $\lambda = w_1 \theta_1 + w_2 \theta_2 + \dots + w_n \theta_n$ 
from (\ref{1form}) to be given by $H(\widetilde{w})$. As already mentioned above, this leads to a canonical
choice of a model for $\Delta^{\stab}$. 

\begin{theo}
	With the above choices we have 
	\[
	\eta_n = (-1)^{n+1 \choose 2}\frac{1}{n!}(dQ)^n \quad \text{$(\mod \Delta_1, \dots, \Delta_n)$}\text{.}
	\]
\end{theo}

\begin{proof} Let us introduce $H_i = h_i \circ P$. 	
	By Proposition \ref{bb.homotopyinverse}, we have to calculate 
	\[
	(\iota_{\infty}(\id))_n = (-1)^n (H\delta)^n(\id) \text{.}
	\]
	In the following calculation, we use the convention to sum over all indices which appear (Einstein's sum
	convention). 
	In addition to the usual Koszul signs rule, we will use the following key facts. 
	\begin{enumerate}[(i)]
		\item \label{n1} All terms involving the composition $h_i h_j$ vanish after summing over all
	indices (\ref{properties} (\ref{p2})).
		\item The operators
			$h_i$ satisfies the Leibniz rule modulo $\Delta_i$ (\ref{properties} (\ref{p1})).
		\item The projection operator $P$ commutes
			with $h_i$ (\ref{properties} (\ref{p3})).
		\item By choice, we have $\lambda = h_iP(\widetilde{w})$. 
	\end{enumerate}
	We start by calculating
	\begin{equation*}
		\begin{split}
		(H\delta)(\id) & = h_iP(d_Q(\id) + \lambda) = P(h_i( Q_x - Q_y) +
		h_i(h_jP(\widetilde{w}))\\
		& =  - P h_i (Q_y) \text{.}
		\end{split}
	\end{equation*}
	Proceeding, we find
	\begin{equation*}
		\begin{split}
			(H\delta)^2(\id)  
			&= -h_j P[ Q_x P(h_i(Q_y)) - P(h_i(Q_y))Q_y + \lambda P(h_i Q_y)]\\
			&= -P [ Q_x P(h_j (h_i(Q_y))) - P(h_j(h_i(Q_y)))Q_y -
			            P(h_i(Q_y))h_j(Q_y) \\
			&\quad+ h_j(\lambda)(h_i Q_y) -
			\lambda(h_j(h_i Q_y)) ] \quad \text{$(\mod \Delta_j)$}\\
			& = P( P(h_i(Q_y)) h_j(Q_y)) \quad \text{$(\mod \Delta_j)$}\text{,}
		\end{split}
	\end{equation*}
	where most of the terms vanish thanks to (\ref{n1}).
	An iteration of this argument leads to the formula
	\[
		\begin{split}
	(H\delta)^n(\id) & = (-1)^n P ( P (\dots P(h_{i_1}(Q_y)) h_{i_2}(Q_y) \dots)
	h_{i_n}(Q_y) )\quad \text{$(\mod \Delta_1, \dots, \Delta_n)$}\\
	& = (-1)^n(-1)^{n+1 \choose 2} \frac{1}{n!} \partial_{i_1}Q \partial_{i_2}Q \cdots
	\partial_{i_n}Q \theta_{i_1} \theta_{i_2} \cdots \theta_{i_n}\quad \text{$(\mod \Delta_1, \dots, \Delta_n)$,}
\end{split}
	\]
	where the last equality follows from Proposition \ref{properties}
	(\ref{p4}) and the Koszul sign interaction of $\theta_i$ with $Q$.
	Incorporating the additional sign $(-1)^n$ leads to the claimed formula.
\end{proof}

Combining this result with Lemma \ref{bulk.lemma}, we obtain an explicit formula
for the boundary-bulk map.

\begin{theo}\label{theorem.bb}
	The boundary-bulk map admits the explicit formula
	\begin{equation}\label{boundary.bulk}
		[\underline{\tr}](\varphi_E): \; \Hom(E,E) \to \Omega_w[n],\; F
		\mapsto (-1)^{n+1 \choose 2}\frac{1}{n!} \tr( F (dQ)^{\wedge n} )
	\end{equation}
	where $Q$ is the twisted differential corresponding to $E$. 
\end{theo}
\begin{proof}
	Note that the sign contribution $n|F|$ cancels by using the cyclic
	symmetry of the graded trace map.
\end{proof}

Note that on cohomology, our formula and the one given in \cite{polishchuk}
produce (up to sign) the same map. The difference is that our map is
well adapted to the Kapustin-Li pairing on the chain level.

\section{Calabi-Yau structure and Topological Quantum Field Theories}
\label{sect.topfield}

\subsection{Topological field theories}

In this section, we explain the relevance of the category $\MF(R,w)$ as a
category of boundary conditions in the context of topological quantum field
theories of various flavors.

It follows from the results in \cite{dyck4} that there exists a
$2$-dimensional framed extended topological field theory in the sense of \cite{lurie} which
maps the trivially-framed point to the category $\MF(R,w)$. Here, we consider $\MF(R,w)$ as an object of an
appropriately defined $(\infty, 2)$-category $\mathcal C$ of $2$-periodic dg categories. 
Indeed, the smoothness and properness of $\MF(R,w)$
established in \cite{dyck4} imply that this category is fully dualizable in
$\mathcal C$. The assertion then follows from \cite[2.4.6]{lurie}.

As first established by Auslander, the triangulated category $[\MF(R,w)]$ admits
a Calabi-Yau structure. This suggests that the category $\MF(R,w)$ will be a
Calabi-Yau object in $\mathcal C$ in the sense of \cite[4.2.6]{lurie}. The
results of this work, will allow us to establish the existence of a Calabi-Yau structure
on the dg category $\MF(R,w)$ explicitly.
In view of \cite[4.2.11]{lurie}, this implies the existence of a $2$-dimensional
oriented extended topological field theory. 

Alternatively, using a theorem of Kontsevich and Soibelman \cite[10.2.2]{kontsevich-2006} (also
cf. \cite{cho-lee} for more details), the results of Section \ref{calabiyau} imply the existence of a minimal $A_{\infty}$ model on which the Calabi-Yau pairing has strictly cyclic
symmetry. This implies the existence of an open-closed field theory in the sense of
\cite[Theorem A]{costello}, where this notion of a strict Calabi-Yau $A_{\infty}$ algebra
is used (see \cite[7.2]{costello}). Using Costello's framework we will explain
how to deduce a Riemann-Roch formula from the existence of the field theory. The
formula presumably agrees with the one recently
established in \cite{polishchuk} (building on the work of \cite{shklyarov}).

\subsection{Calabi-Yau dg algebras}\label{calabiyau}

To put us into context, recall that a Frobenius algebra is a unital $k$-algebra $A$
together with a non-degenerate pairing
\[
A \otimes_k A \lra k,\; a \otimes b \mapsto \left< a,b \right>
\]
which satisfies
\begin{align}\label{cyclic}
\left<ab,c\right> = \left<bc,a\right> = \left<ca,b\right>
\end{align}
for all elements $a,b,c$ in $A$. Equivalently, we could formulate the definition in terms of the trace map
\[
\tr: \; A \to k, \; a \mapsto \left< a, 1 \right>\text{,}
\]
the corresponding pairing is then recovered as $\left< a,b\right> = \tr(ab)$.
The cyclic symmetry can be reformulated by saying that there exists a commutative diagram
\begin{align}\label{frobp}
\xymatrix{
A \otimes_k A \ar[r]^(.55){\left<-,-\right>} \ar[d] & k\\
A \otimes_{A \otimes_k A^{\op}} A \ar@{.>}[ur]
}
\end{align}
or, in terms of the trace map, 
\begin{align}\label{frob}
\xymatrix{
A \ar[r]^(.55){\tr} \ar[d] & k\\
A \otimes_{A \otimes_k A^{\op}} A\text{.} \ar@{.>}[ur]
}
\end{align}
Trying to generalize this notion to the context of differential graded algebras, we
would certainly start by requiring the existence of a pairing
\[
A \otimes_k A \lra k,\;  a \otimes b \mapsto \left<a,b\right>
\]
which is homologically non-degenerate. 
Whatever the notion of cyclicity should be, we would like it to be invariant under weak equivalences.
To achieve this desideratum, we require the existence of a commutative diagram
\begin{align}\label{dgfrob}
\xymatrix{
A  \ar[r]^(0.55){\tr} \ar[d]^{\beta} & k\\
A \otimes_{A \otimes_k A^{\op}}^{L} A \ar@{.>}[ur]_(0.6){\tr^\infty}
}
\end{align}
where $\beta$ is the boundary-bulk map from (\ref{boundary-bulk}) in
the case where the dg category $T$ has a single object with endomorphism dga $A$.
Observe, that the existence of the lift of the map $\tr$ in diagram (\ref{frob}) is a
property of the pairing. In contrast, specifying the map $\tr^\infty$ in diagram
(\ref{dgfrob}) requires the specification of additional structure, corresponding
to a coherent system of homotopies between the expressions appearing in
(\ref{cyclic}).
Indeed, these higher homotopies become explicit by using the cyclic bar
construction $\C(A)$ as a model for the (Hochschild) complex 
$A \otimes_{A \otimes_k A^{\op}}^{L}A$. Within this model, the map $\beta$ is
simply the inclusion of $A$ as a subcomplex of $\C(A)$. Constructing a map $\tr^{\infty}$ such that
(\ref{dgfrob}) commutes thus amounts to providing an extension of the trace map
on $A$ to one on $\C(A)$. 

The existence of the commutative diagram (\ref{dgfrob}) does not suffice to obtain 
an oriented extended field theory. Indeed, assuming the existence of such a field theory, the Hochschild
complex of $A$ will be assigned to the circle. The symmetries of the circle will
therefore act on the Hochschild complex and the trace map is seen to be
equivariant with respect to this action. On the level of chain complexes, the action
of the circle on the Hochschild complex, translates into the action of Connes'
$B$-operator. The equivariance condition amounts to providing a lift of the
trace map from the Hochschild complex to the cyclic complex.
\begin{align}\label{dgcalabi}
\xymatrix{
A  \ar[r]^(0.55){\tr} \ar[d]^{\beta} & k \ar[dd]\\
A \otimes_{A \otimes_k A^{\op}}^{L} A \ar[d] \ar@{.>}[ur]_(0.6){\tr^\infty}\\
(A \otimes_{A \otimes_k A^{\op}}^{L} A)_{S^1}
\ar@{.>}[r]_(0.7){\tr^\infty_{S^1}} & k_{S^1}
}
\end{align}
Here, we may choose Connes' cyclic complex
\[
(A \otimes_{A \otimes_k A^{\op}}^{L} A)_{S^1} \simeq \CC(A) := (\C(A)[u^{-1}], b + u B)
\]
as an explicit model in which the complex $\C(A)$ appears as a subcomplex. A dg
algebra $A$ with a homologically non-degenerate pairing $\tr$ together with a lift
$\tr^{\infty}_{S^1}$ to the cyclic complex is called a Calabi-Yau dg algebra. This
structure was already studied by Kontsevich and Soibelman in \cite{kontsevich-2006}. 
Interestingly, as proved in \cite{kontsevich-2006}, one
can always strictify a Calabi-Yau structure by passing to an appropriate minimal
$A_\infty$-model of $A$ on which the pairing becomes strictly cyclic. This strict notion of a 
Calabi-Yau $A_\infty$ algebra is used in \cite{costello}. Costello proves that every strict Calabi-Yau $A_\infty$-algebra
defines an open topological conformal field theory (TCFT) which canonically
extends to a universal open-closed TCFT.

We outline how to construct the Calabi-Yau dg algebra which will provide an
open-closed TCFT associated to the dg category $T = \MF(R,w)$ of matrix
factorizations.

First, we apply Theorem 4.2 in \cite{dyck4} which allows us to restrict our study 
to the endomorphism dg algebra $A = T(E,E)$ of a single matrix factorization $E$ in $\MF(R,w)$. 
The Kapustin-Li formula provides us with a trace map 
\begin{align}\label{kltrace}
	\tr_{KL} : A \lra k,\; F \mapsto 
	(-1)^{n+1 \choose 2}\frac{1}{n!} \Res \left[ \begin{array}{c} 
		\tr( F (dQ)^{\wedge n} )\\
		\partial_1 w, \partial_2 w,\; \cdots , \partial_n w \end{array}  \right] 
\end{align}
which, by Theorem \ref{kl}, induces a homologically non-degenerate pairing on
$A$. We have to show that this trace map is part of a Calabi-Yau structure on $A$, in other
words, we have to extend the trace map to a map $\tr^{\infty}_{S^1}$ on the cyclic complex.
By Theorem 5.7 in \cite{dyck4}, the Hochschild complex $\C(A)$ is quasi-isomorphic to 
the Milnor algebra $\Omega_w$ concentrated in degree given by the parity of $n$.
By the degeneration of the Hochschild-to-cyclic spectral sequence (cf.
\cite{dyck4}), we further know that the 
cyclic complex $\CC(A)$ is quasi-isomorphic to $\Omega_w[u^{-1}]$ concentrated
in the same degree. Thus, the diagram (\ref{dgcalabi}) specializes to 
\begin{align}\label{klcalabi}
\xymatrix{
A  \ar[r]^(0.55){\tr_{KL}} \ar[d]^{\beta} & k \ar[dd]\\
\Omega_w \ar[d] \ar@{.>}[ur]_{\tr^\infty}\\
\Omega_w[u^{-1}] \ar@{.>}[r]_{\tr^\infty_{S^1}} & k[u^{-1}]\text{,}
}
\end{align}
where we omitted the shifts by $[n]$.
The map $\beta$ coincides with the boundary-bulk map
$[\underline{\tr}](\pi_E)$ studied in Section \ref{sect.bb}. By Theorem
\ref{theorem.bb}, we have the formula
\[
\beta(F) = (-1)^{n+1 \choose 2}\frac{1}{n!}\tr( F (dQ)^{\wedge n} ) \text{.}
\]
Therefore, we can complete diagram (\ref{klcalabi}) by letting
\[
\tr^\infty(\omega) := 
	\Res \left[ \begin{array}{c} 
		\omega \\
		\partial_1 w, \partial_2 w,\; \cdots , \partial_n w \end{array}  \right] 
\]
and defining $\tr^\infty_{S^1}$ by extending $k[u^{-1}]$-linearly.
This provides $A$ with a Calabi-Yau structure. In particular, the above
mentioned result due to Kontsevich-Soibelman assures the existence of minimal
strictly cyclic models of $A$. In \cite{carqueville}, the author develops and implements
an algorithm to explicitly calculate such cyclic minimal models.

\subsection{Riemann-Roch formula}
\label{sect.rr}

Finally, we sketch how to deduce a Riemann-Roch formula from the existence of a field theory. Using
\cite[10.2.2]{kontsevich-2006}, we pass to a strictly cyclic minimal $A_\infty$-model of
$\MF(R,w)$. Here, we restrict our attention to a
direct sum of finitely many objects in $\MF(R,w)$ such that the method explained
in the previous section becomes applicable. 

By a $\Zt$-graded variant of Costello's Theorem A \cite{costello}, we obtain the existence of a canonical
open-closed field theory associated to $\MF(R,w)$. 
Within this field theory, the boundary-bulk map $[\underline{\tr}](\pi_E)$ is the map of chain complexes
associated to the cobordism visualized in Figure \ref{fig.bb}.
\begin{figure}[htp]
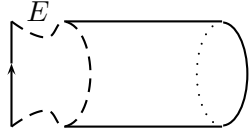


\centering

\psset{unit=0.7cm}
\pspicture(0,0)(5,3)

\put(0,1.5){
\psline(0,1)(0,0.2)
\psline{<-}(0,0.2)(0,-1)
\pscurve[linestyle=dashed](0,1)(0.7,0.7)(1,1)(1.5,0)(1,-1)(0.7,-0.7)(0,-1)

\psclip{\psframe[linestyle=none](3,-1.5)(4,1.5)}
	\psellipse[linestyle=dotted](4,0)(0.5,1)
\endpsclip
\psclip{\psframe[linestyle=none](4,-1.5)(5,1.5)}
	\psellipse[linestyle=solid](4,0)(0.5,1)
\endpsclip
\psline(1,1)(4,1)
\psline(1,-1)(4,-1)
\rput(0.5,1.2){$E$}
}
\endpspicture

\caption{The boundary-bulk map}\label{fig.bb}
\end{figure}

We define the Chern character of $E$ to be 
\[
\ch(E) = [\underline{\tr}](\pi_E)(\id_E) \in \HH_0(\MF(R,w)) \text{.}
\]
For matrix factorizations $E$, $F$ in $\MF(R,w)$ we define the $\Zt$-graded Euler
characteristic
\[
\chi \Hom(E,F) = \dim \H^0(\Hom(E,F)) - \dim \H^1(\Hom(E,F)) \text{.}
\]
The field theory corresponding to $\MF(R,w)$ assigns a scalar $\lambda \in k$ to the
cobordism drawn in Figure \ref{fig.sphere}.
\begin{figure}[htp]
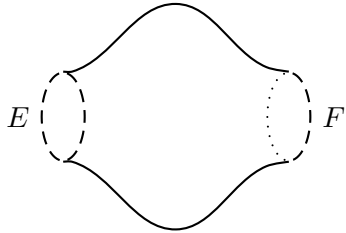


\centering

\psset{unit=0.6cm}
\pspicture(0,-3)(7,3)

\psellipse[linestyle=dashed](1,0)(0.5,1)
\psclip{\psframe[linestyle=none](6,-1.5)(8,1.5)}
	\psellipse[linestyle=dashed](6,0)(0.5,1)
\endpsclip
\psclip{\psframe[linestyle=none](4,-1.5)(6,1.5)}
	\psellipse[linestyle=dotted](6,0)(0.5,1)
\endpsclip
\pscurve(1,1)(1.2,1)(3.5,2.5)(5.5,1.1)(6,1)
\pscurve(1,-1)(1.2,-1)(3.5,-2.5)(5.5,-1.1)(6,-1)
\rput(7,0){$F$}
\rput(0,0){$E$}

\endpspicture

\caption{Twice punctured sphere}\label{fig.sphere}
\end{figure}
which is a sphere with two disks removed, where the dashed lines indicate free boundaries labelled by the objects $E$, $F$.

The field theory formalism allows us to calculate this number in two different
ways, by decomposing the above punctured sphere. Consider first the
decomposition illustrated in Figure \ref{fig.dec1}.
Interpreting all components appropriately, this yields the formula
\[
\lambda = \left< \ch(E), \ch(F) \right> \text{.}
\]


\begin{figure}[htp]
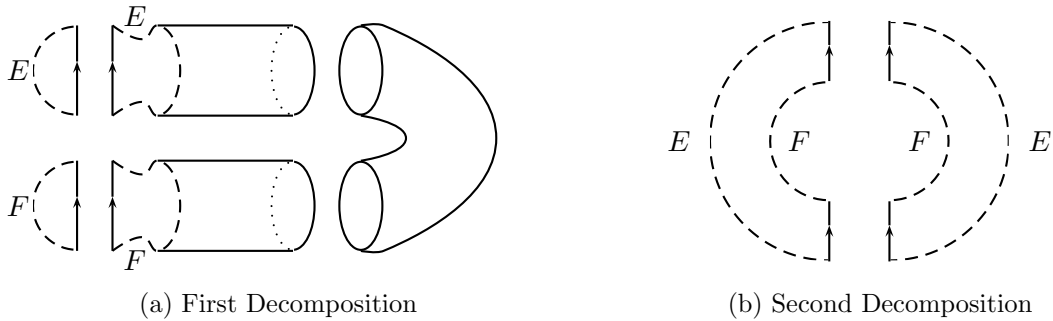


\hfill
\begin{minipage}[t]{.5\textwidth}

\psset{unit=0.6cm}
\pspicture(-1,-3)(10,3)

\put(0,1.5){
\psclip{\psframe[linestyle=none](0,-2)(1,2)}
\psellipse[linestyle=dashed](1,0)(1,1)
\endpsclip
\psline(1,1)(1,0.2)
\psline{<-}(1,0.2)(1,-1)
\rput(-0.3,0){$E$}
}

\put(0,-1.5){
\psclip{\psframe[linestyle=none](0,-2)(1,2)}
\psellipse[linestyle=dashed](1,0)(1,1)
\endpsclip
\psline(1,1)(1,0.2)
\psline{<-}(1,0.2)(1,-1)
\rput(-0.3,0){$F$}
}

\put(2,1.5){
\psline(0,1)(0,0.2)
\psline{<-}(0,0.2)(0,-1)
\pscurve[linestyle=dashed](0,1)(0.7,0.7)(1,1)(1.5,0)(1,-1)(0.7,-0.7)(0,-1)

\psclip{\psframe[linestyle=none](3,-1.5)(4,1.5)}
	\psellipse[linestyle=dotted](4,0)(0.5,1)
\endpsclip
\psclip{\psframe[linestyle=none](4,-1.5)(5,1.5)}
	\psellipse[linestyle=solid](4,0)(0.5,1)
\endpsclip
\psline(1,1)(4,1)
\psline(1,-1)(4,-1)
\rput(0.5,1.2){$E$}
}

\put(2,-1.5){
\psline(0,1)(0,0.2)
\psline{<-}(0,0.2)(0,-1)
\pscurve[linestyle=dashed](0,1)(0.7,0.7)(1,1)(1.5,0)(1,-1)(0.7,-0.7)(0,-1)

\psclip{\psframe[linestyle=none](3,-1.5)(4,1.5)}
	\psellipse[linestyle=dotted](4,0)(0.5,1)
\endpsclip
\psclip{\psframe[linestyle=none](4,-1.5)(5,1.5)}
	\psellipse[linestyle=solid](4,0)(0.5,1)
\endpsclip
\psline(1,1)(4,1)
\psline(1,-1)(4,-1)
\rput(0.5,-1.2){$F$}
}

\put(7.5,1.5){
\psellipse(0,0)(0.5,1)
\psellipse(0,-3)(0.5,1)
\pscurve(0,1)(0.5,1)(3,-1.5)(0.5,-4)(0,-4)
\pscurve(0,-1)(1,-1.5)(0,-2)
}

\endpspicture

\subcaption{First Decomposition}\label{fig.dec1}
\end{minipage}%
\begin{minipage}[t]{.5\textwidth}

\psset{unit=0.8cm}
\pspicture(-2,-2.2)(3,2)

\put(0,0){
\psclip{\psframe[linestyle=none](0,-2)(2,2)}
	\psellipse[linestyle=dashed](2,0)(2,2)
\endpsclip
\psclip{\psframe[linestyle=none](0,-2)(2,2)}
	\psellipse[linestyle=dashed](2,0)(1,1)
\endpsclip
\psline(2,2)(2,1.6)
\psline{<-}(2,1.6)(2,1)
\psline(2,-1)(2,-1.4)
\psline{<-}(2,-1.4)(2,-2)
\rput(-0.5,0){$E$}
\rput(1.5,0){$F$}
}

\put(1,0){
\psclip{\psframe[linestyle=none](2,-2)(4,2)}
	\psellipse[linestyle=dashed](2,0)(2,2)
\endpsclip
\psclip{\psframe[linestyle=none](2,-2)(4,2)}
	\psellipse[linestyle=dashed](2,0)(1,1)
\endpsclip
\psline(2,2)(2,1.6)
\psline{<-}(2,1.6)(2,1)
\psline(2,-1)(2,-1.4)
\psline{<-}(2,-1.4)(2,-2)
\rput(2.5,0){$F$}
\rput(4.5,0){$E$}

}

\endpspicture

\subcaption{Second Decomposition}\label{fig.dec2}
\end{minipage}%
\caption{Two Decompositions}
\end{figure}

Secondly, we first flatten the punctured sphere into the plane and then
decompose as illustrated in Figure \ref{fig.dec2}.
This yields the formula
\[
\lambda = \tr (\id_{\H^*(\Hom(E,F))}) = \chi \Hom(E,F) \text{,}
\]
where $\tr$ denotes the graded trace map.
Thus, we obtain the Hirzebruch-Riemann-Roch formula 
\[
\chi \Hom(E,F) = \left< \ch(E), \ch(F) \right> \text{.}
\]
To compare to the formula obtained in \cite[4.1.4]{polishchuk}, one had to
calculate the pairing $\left<-,-\right>$ on the Hochschild homology produced by
the field theory (which depends on our choice of the Calabi-Yau structure
$\tr^{\infty}_{S^1}$) and compare it to the canonical pairing calculated in \cite{polishchuk}.

\bibliographystyle{halpha}
\bibliography{biblio}

\begin{thebibliography}{Kun08}

\bibitem[Aus78]{auslander}
M.~Auslander.
\newblock Functors and morphisms determined by objects.
\newblock {\em Representation theory of algebras ({P}roc. {C}onf., {T}emple
  {U}niv., {P}hiladelphia, {P}a., 1976) Lecture Notes in Pure Appl. Math. Vol.
  37}, pages 1--244, 1978.

\bibitem[BH93]{brunsherzog}
W.~Bruns and J.~Herzog.
\newblock {\em Cohen-{M}acaulay rings}, volume~39 of {\em Cambridge Studies in
  Advanced Mathematics}.
\newblock Cambridge University Press, Cambridge, 1993.

\bibitem[Buc86]{buchweitz}
R.-O. Buchweitz.
\newblock {Maximal Cohen-Macaulay modules and Tate-Cohomology over Gorenstein
  rings}.
\newblock {\em preprint}, 1986.

\bibitem[Car09]{carqueville}
N.~Carqueville.
\newblock {Matrix factorisations and open topological string theory}.
\newblock {\em JHEP}, 0907:005, 2009.

\bibitem[CL10]{cho-lee}
C.-H. Cho and S.~Lee.
\newblock {Notes on Kontsevich-Soibelman's theorem about cyclic A-infinity
  algebras}, 2010, arXiv:1002.3653v2.

\bibitem[Cos07]{costello}
K.~Costello.
\newblock {Topological conformal field theories and Calabi-Yau categories}.
\newblock {\em Adv. Math.}, 210(1):165--214, 2007.

\bibitem[Cra04]{crainic}
M.~Crainic.
\newblock On the perturbation lemma, and deformations.
\newblock {\em arXiv:math.AT/0403266}, 2004.

\bibitem[Dyc09]{dyck4}
T.~Dyckerhoff.
\newblock {Compact generators in categories of matrix factorizations}, 2009,
  arXiv:math/0904.4713.

\bibitem[Eis80]{eisenbud}
D.~Eisenbud.
\newblock Homological algebra on a complete intersection, with an application
  to group representations.
\newblock {\em Trans. Amer. Math. Soc.}, 260:35--64, 1980.

\bibitem[GH94]{griffiths}
P.~Griffiths and J.~Harris.
\newblock {\em Principles of algebraic geometry}.
\newblock Wiley Classics Library. John Wiley \& Sons Inc., New York, 1994.
\newblock Reprint of the 1978 original.

\bibitem[Har66]{residuesandduality}
R.~Hartshorne.
\newblock {\em Residues and duality}.
\newblock Lecture notes of a seminar on the work of A. Grothendieck, given at
  Harvard 1963/64. With an appendix by P. Deligne. Lecture Notes in
  Mathematics, No. 20. Springer-Verlag, Berlin, 1966.

\bibitem[Har67]{hartlocal}
R.~Hartshorne.
\newblock {\em Local cohomology}, volume 1961 of {\em {A seminar given by A.
  Grothendieck, Harvard University, Fall}}.
\newblock Springer-Verlag, Berlin, 1967.

\bibitem[KR04]{kapustin}
A.~Kapustin and L.~Rozansky.
\newblock On the relation between open and closed topological strings.
\newblock {\em Comm. Math. Phys.}, 252(1-3):393--414, 2004.

\bibitem[KS06]{kontsevich-2006}
M.~Kontsevich and Y.~Soibelman.
\newblock {Notes on A-infinity algebras, A-infinity categories and
  non-commutative geometry. I}, 2006, arXiv:math/0606241.

\bibitem[Kun86]{kunz.kaehler}
Ernst Kunz.
\newblock {\em {K\"ahler differentials}}.
\newblock Advanced Lectures in Mathematics. Friedr. Vieweg \& Sohn,
  Braunschweig, 1986.

\bibitem[Kun08]{kunz}
E.~Kunz.
\newblock {\em Residues and duality for projective algebraic varieties},
  volume~47 of {\em University Lecture Series}.
\newblock American Mathematical Society, Providence, RI, 2008.
\newblock {With the assistance of and contributions by David A. Cox and Alicia
  Dickenstein}.

\bibitem[Lip84]{lipman}
J.~Lipman.
\newblock Dualizing sheaves, differentials and residues on algebraic varieties.
\newblock {\em Ast\'erisque}, (117):ii+138, 1984.

\bibitem[Loo84]{looijenga}
E.~J.~N. Looijenga.
\newblock {\em {Isolated singular points on complete intersections}}, volume~77
  of {\em London Mathematical Society Lecture Note Series}.
\newblock Cambridge University Press, Cambridge, 1984.

\bibitem[Lur09]{lurie}
J.~Lurie.
\newblock {On the Classification of Topological Field Theories}, 2009,
  arXiv:math/0905.0465.

\bibitem[Mur09]{murfet-2009}
D.~Murfet.
\newblock {Residues and duality for singularity categories of isolated
  Gorenstein singularities}, 2009, arXiv:math/0912.1629.

\bibitem[PV10]{polishchuk}
A.~Polishchuk and A.~Vaintrob.
\newblock {Chern characters and Hirzebruch-Riemann-Roch formula for matrix
  factorizations}, 2010, arXiv:math/1002.2116.

\bibitem[Seg09]{segal}
E.~Segal.
\newblock The closed state space of affine {L}andau-{G}inzburg {B}-models,
  2009, arXiv:math/0904.1339.

\bibitem[Shk07]{shklyarov}
D.~Shklyarov.
\newblock {Hirzebruch-Riemann-Roch theorem for DG algebras}, 2007,
  arXiv:0710.1937.

\bibitem[To{\"e}06]{toen.lectures}
B.~To{\"e}n.
\newblock {Lectures on DG-categories}, 2006.
\newblock available at \url{http://www.math.univ-toulouse.fr/~toen/swisk.pdf}.

\bibitem[To{\"e}07]{toen.morita}
B.~To{\"e}n.
\newblock The homotopy theory of {$dg$}-categories and derived {M}orita theory.
\newblock {\em Invent. Math.}, 167(3):615--667, 2007.

\bibitem[Yos90]{yoshino}
Y.~Yoshino.
\newblock {\em {Cohen---Macaulay modules over Cohen---Macaulay Rings}}, volume
  146 of {\em London Mathematical Society Lecture Notes Series}.
\newblock Cambridge University Press, 1990.

\end{thebibliography}

\noindent
Tobias Dyckerhoff: Department of Mathematics, University of Pennsylvania\\
email: \texttt{tdyckerh@math.upenn.edu}\\[1ex]
Daniel Murfet: Hausdorff Center for Mathematics, University of Bonn\\
email: \texttt{daniel.murfet@math.uni-bonn.de}

\end{document}